\documentclass{amsart}
\usepackage[margin=1in]{geometry}
\usepackage{amsmath}
\usepackage{amsfonts}
\usepackage{amssymb}
\usepackage{amsthm}
\usepackage{cancel}
\usepackage{tikz-cd}

\newcommand{\RR}{\mathbb{R}}

\newcommand{\ZZ}{\mathbb{Z}}
\newcommand{\QQ}{\mathbb{Q}}
\newcommand{\OO}{\mathcal{O}}

\theoremstyle{theorem}
\newtheorem{theorem}{Theorem}[section]
\newtheorem{corollary}{Corollary}[section]
\newtheorem{lemma}{Lemma}[section]

\begin{document}
\title{Orders of Quaternion Algebras with Involution}

\author{Arseniy Sheydvasser}

\maketitle

\begin{abstract}
We introduce the notion of maximal orders over quaternion algebras with orthogonal involution and give a classification over local and global fields. Over local fields, we show that there is a correspondence between maximal and/or modular lattices and orders closed under the involution.
\end{abstract}

\section{Introduction:}

The study of maximal orders of quaternion algebras has a rich history, going back to the work of Hurwitz in the 1800s. Finding such orders is essentially a solved problem. Indeed, an explicit algorithm is given in \cite{IvanyosRonyai} that constructs maximal orders in a semisimple algebra over an algebraic number field in polynomial time, with a simpler algorithm given in \cite{VoightArticle} for the special case of quaternion algebras. It is a straight-forward exercise to check that any order of a quaternion algebra is closed under the standard involution, i.e. quaternion conjugation. However, given an arbitrary involution $\ddagger$ on a quaternion algebra $H$ and an order $\OO \subset H$, it need not be the case that $\OO = \OO^\ddagger$. In the special case where $\OO = \OO^\ddagger$, we shall call $\OO$ a $\ddagger$-\emph{order}. In this paper, we characterize $\ddagger$-orders over local and global fields---in the local case, we show that either the quaternion algebra is a division algebra or there is a correspondence between maximal or modular lattices and $\ddagger$-orders (see Theorems \ref{RamifiedCase}, \ref{NormalCount}, and \ref{StrangeCount}).

The interest in this subject comes from the problem of constructing certain arithmetic sphere packings in $\RR^3$. In $\RR^2$, one can construct \cite{Stange1} \cite{Stange2} circle packings by considering the action of a Bianchi group $SL(2,\OO_K)$ on the real line, where $\OO_K$ is the ring of integers of some imaginary quadratic field $K$. This same method can be used to produce interesting sphere packings in $\RR^3$, but requires replacing the Bianchi group $SL(2,\OO)$ with an appropriate analog $SL^\ddagger(2,\OO)$, where $\OO$ is a maximal $\ddagger$-order (see \cite{Self}).

For this reason, we shall be primarily interested in maximal $\ddagger$-orders---that is, $\ddagger$-orders not properly contained inside any other $\ddagger$-order. Since any order is contained inside a maximal order, it is easy to see that any maximal $\ddagger$-order must be of the form $\OO \cap \OO^\ddagger$, where $\OO$ is maximal. Therefore, all maximal $\ddagger$-orders are Eichler orders---that is, the intersection of two (not necessarily distinct) maximal orders. However, it is not the case that all Eichler orders of the form $\OO \cap \OO^\ddagger$ are maximal.

As a first example, consider the two maximal orders

	\begin{align*}
	\OO_1 &= \ZZ \oplus \ZZ i \oplus \ZZ j \oplus \ZZ \frac{5 + 5i + 3j - ij}{10} \subset \left(\frac{-1,-5}{\QQ}\right)\\
	\OO_2 &= \ZZ \oplus 9\ZZ i \oplus \ZZ j \oplus \ZZ\frac{90 + 385i - 63j + ij}{90} \subset \left(\frac{-1,-5}{\QQ}\right).
	\end{align*}
	
\noindent Define an involution on $\left(\frac{-1,-5}{\QQ}\right)$ by

	\begin{align*}
	(w + xi + yj + zij)^\ddagger &= w + xi + yj - zij.
	\end{align*}
	
\noindent Then it is an easy computation that

	\begin{align*}
	\OO_1 \cap \OO_1^\ddagger &= \ZZ \oplus \ZZ i \oplus \ZZ j \oplus \ZZ\frac{1 + i + j + ij}{2} \\
	\OO_2 \cap \OO_2^\ddagger	&= \ZZ \oplus 9\ZZ i \oplus j \oplus 9\ZZ \frac{1 + i + j + ij}{2},
	\end{align*}
	
\noindent which shows that $\OO_2 \cap \OO_2^\ddagger$ is not a maximal $\ddagger$-order. What is less evident is that $\OO_1 \cap \OO_1^\ddagger$ is a maximal $\ddagger$-order. As we shall describe presently, for an Eichler order $\OO \cap \OO^\ddagger$ to be maximal, there is a necessary and sufficient restriction on the discriminant of the order.

Let $F$ be a local or global field of characteristic not 2, and $\mathfrak{o}$ its ring of integers. If $H$ is quaternion algebra over $F$ with an orthogonal involution $\ddagger$, we denote by $disc(\ddagger) \in F^\times/\left(F^\times\right)^2$ the discriminant of $\ddagger$ (see section \ref{Preliminaries} for definitions). Unlike the discriminant $disc(H)$ of the quaternion algebra, this is not an ideal of $\mathfrak{o}$, but we can associate an ideal to $disc(\ddagger)$ by defining a map

	\begin{align*}
	\iota: F^\times/\left(F^\times\right)^2 &\rightarrow \left\{\text{square-free ideals of } \mathfrak{o}\right\} \\
	[\lambda] &\mapsto \bigcup_{\lambda \in [\lambda] \cap \mathfrak{o}} \lambda \mathfrak{o}.
	\end{align*}
	
With these conventions, our main result is the following simple statement.

\begin{theorem}\label{MainTheorem}
Given a quaternion algebra $H$ over a local or global field $F$, the maximal $\ddagger$-orders of $H$ correspond to Eichler orders of the form $\OO \cap \OO^\ddagger$ with discriminant

	\begin{align*}
	disc(H) \cap \iota(disc(\ddagger)).
	\end{align*}
\end{theorem}

Note that it is sufficient to prove that all maximal $\ddagger$-orders have discriminants of the desired form---if $\OO \cap \OO^\ddagger$ is an Eichler order, it must be contained inside a maximal $\ddagger$-order, and if their discriminants match, they must be the same.

In the process of proving Theorem \ref{MainTheorem}, we establish a far more explicit classification of maximal $\ddagger$-orders over local fields. In particular, we show that if $H = End(V)$, where $V$ is a 2-dimensional quadratic space over a local field, then there is a correspondence between either maximal or modular lattices in $V$ and maximal $\ddagger$-orders of $H$. 

\subsection*{Acknowledgements} The author would like to thank Asher Auel, Alex Kontorovich and John Voight for their helpful suggestions and insights.

\section{Preliminaries:}\label{Preliminaries}

We recall some basic facts and conventions about involutions over central simple algebras. Let $H$ be a central simple algebra of dimension $n^2$ over a field $F$ (with $char(F) \neq 2$). Recall that an involution (of the first type) is an $F$-linear map $\ddagger$ on $H$ such that

\hskip0.1in

	\begin{enumerate}
		\item $(xy)^\ddagger = y^\ddagger x^\ddagger \ \left(\forall x,y \in H\right)$, and
		\item $\ddagger^2 = id$.
	\end{enumerate}
	
\hskip0.1in
	
A homomorphism of two algebras with involution $(H_1, \ddagger_1)$ and $(H_2, \ddagger_2)$ is an $F$-algebra homomorphism $\phi: H_1 \rightarrow H_2$ such that

	\begin{align*}
	\phi(\sigma^{\ddagger_1}) = \phi(\sigma)^{\ddagger_2} \ \ \ (\forall \sigma \in H_1).
	\end{align*}
	
\noindent This is the correct notion of homomorphism for our purposes, since isomorphisms of algebras with involutions

\hskip0.1in

	\begin{enumerate}
		\item send $\ddagger$-orders to $\ddagger$-orders,
		\item send Eichler orders of the form $\OO \cap \OO^\ddagger$ to Eichler orders of the same form, and
		\item preserve the discriminant.
	\end{enumerate}
	
\hskip0.1in
	
\noindent Note that $H$ decomposes as a vector space as $H^+ \oplus H^-$, where

	\begin{align*}
	H^+ &= \left\{x \in H \middle| x^\ddagger = x\right\} \\
	H^- &= \left\{x \in H \middle| x^\ddagger = -x\right\}.
	\end{align*}

\noindent It is well known that involutions on $H$ split into two basic types:

\hskip0.1in

	\begin{enumerate}
		\item \emph{symplectic involutions}, for which $\dim H^+ = \frac{n(n-1)}{2}$ and $\dim H^- = \frac{n(n+1)}{2}$, and
		\item \emph{orthogonal involutions}, for which $\dim H^+ = \frac{n(n+1)}{2}$ and $\dim H^- = \frac{n(n-1)}{2}$.
	\end{enumerate}
	
\hskip0.1in
	
\noindent In the particular case where $H$ is a quaternion algebra, the only symplectic involution is the standard involution of quaternion conjugation---all other involutions are orthogonal.

For a given orthogonal involution $\ddagger$, we define the group

	\begin{align*}
	GO(H, \ddagger) &= \left\{x \in H^\times \middle| x^\ddagger x \in F \right\} \\
	&= \left\{x \in H^\times \middle| x^\ddagger = \pm \overline{x} \right\}.
	\end{align*}
	
\noindent There is an exact sequence

\begin{equation*}
\begin{tikzcd}[row sep=tiny]
0 \arrow[r] & \{\pm 1\} \arrow[r] & GO(H,\ddagger) \arrow[r] & \left\{\substack{\text{automorphisms} \\ \text{of } (H, \ddagger)}\right\} \arrow[r] & 0 \\
& & x \arrow[r, mapsto] & \left(\sigma \mapsto x\sigma x^{-1}\right) &
\end{tikzcd},
\label{ExactSequence}
\end{equation*}

\noindent and so we simply identify $PGO(H, \ddagger) = GO(H, \ddagger)/\{\pm 1\}$ with the automorphism group of $(H,\ddagger)$.

Finally, the discriminant of an orthogonal involution is defined as

	\begin{align*}
	disc(\ddagger) = nrd(h) \cdot \left(F^\times\right)^2 \in F^\times/\left(F^\times\right)^2,
	\end{align*}
	
\noindent where $h$ is any element of $H^-$, and $nrd(h) = \overline{h}h$ is the standard norm map. Orthogonal involutions are classified by the discriminant---that is, given a quaternion algebra $H$ and two orthogonal involutions $\ddagger_1, \ddagger_2$ on $H$ such that $disc(\ddagger_1) = disc(\ddagger_2)$, there is an isomorphism $\phi: (H, \ddagger_1) \rightarrow (H, \ddagger_2)$.

\section{Localization:}

We note that to prove Theorem \ref{MainTheorem}, it is in fact sufficient to prove it for local fields. To see this, $F$ be a global field, let $\Omega$ denote the set of places of $F$, and $\Omega_f, \Omega_\infty$ denote the finite and infinite places of $F$ respectively. Given a quaternion algebra with involution $(H, \ddagger)$, we can consider the localizations

	\begin{align*}
	H_\nu &= H \otimes_F F_\nu \\
	\left(h \otimes t\right)^{\ddagger_\nu} &= h^\ddagger \otimes t
	\end{align*}

\noindent for any $\nu \in \Omega$. We can also define

	\begin{align*}
	\OO_\mathfrak{p} &= \OO \otimes_\mathfrak{o} \mathfrak{o}_\mathfrak{p}
	\end{align*}
	
\noindent for any $\mathfrak{p} \in \Omega_f$. It is easy to check that localization sends $\ddagger$-orders to $\ddagger_{\mathfrak{p}}$-orders, sends Eichler orders to Eichler orders, preserves maximality, and

	\begin{align*}
	disc(H) &= \prod_{\nu \in \Omega} disc(H_\nu) \\
	\iota(disc(\ddagger)) &= \prod_{\nu \in \Omega} \iota(disc(\ddagger_{\nu})).
	\end{align*}
	
\noindent Therefore, if Theorem \ref{MainTheorem} holds for all of the localizations $H_\mathfrak{p}$, it must hold for $H$ itself. Thus, henceforth we shall assume that $F$ is a local field over a place $\mathfrak{p}$ with uniformizer $\pi$.

For any $x \in F$, we shall denote by $ord_\mathfrak{p}(x)$ the smallest integer $n$ such that $x \in \mathfrak{p}^n$.

\section{The Division Algebra Case:}

For any quaternion algebra over a field $F$, it is either a division algebra, or is isomorphic to $Mat(2,F)$ (in which case it is called split). We start by considering the division algebra case, which turns out to be especially simple.

	\begin{theorem}\label{RamifiedCase}
	Let $F$ be a local field, and suppose $H$ is a division algebra. Then there is a unique maximal $\ddagger$-order, given by
	
		\begin{align*}
		\OO &= \left\{h \in H \middle| nrd(h) \in \mathfrak{o}\right\}.
		\end{align*}
		
	\noindent The order $\OO$ is also maximal and furthermore,
	
		\begin{align*}
		disc(\OO) &= \mathfrak{p} = disc(H) \cap \iota(disc(\ddagger)).
		\end{align*}
	\end{theorem}
	
	\begin{proof}
	It is well known that $\OO$ is the unique maximal order of $H$. Consequently, it must be the unique maximal $\ddagger$-order. It is also easy to check that
	
	\begin{align*}
	disc(\OO) &= \mathfrak{p} = disc(H).
	\end{align*}
	
	Since $\iota(disc(\ddagger))$ is a square-free ideal in a local field, it is either $\mathfrak{o}$ or $\mathfrak{p}$. In either case, the theorem holds.
	\end{proof}
	
\section{Split Quaternion Algebras:}

It remains to prove Theorem \ref{MainTheorem} for the case where $H$ is split. We have that $H \cong End(V)$, where $V$ is any 2-dimensional vector space over $F$. We start with a basic observation that we shall use for easy computations of the discriminant.

\begin{lemma}\label{ExplicitAlgebra}
Let $H$ be a split quaternion algebra over a local field $F$ with involution $\ddagger$. Choose $\lambda \in disc(\ddagger)$ such that $\lambda$ is a generator of $\iota(disc(\ddagger))$. Then $(H, \ddagger)$ is isomorphic to $Mat(2,F)$ with the involution $\ddagger_\lambda$ given by

	\begin{align*}
	\begin{pmatrix} a & b \\ c & d \end{pmatrix}^{\ddagger_\lambda} = \begin{pmatrix} a & c/\lambda \\ b \lambda & d \end{pmatrix}.
	\end{align*}
\end{lemma}

\begin{proof}
Since $H \cong Mat(2,F) = End(F^2)$ as algebras, and involutions on algebras are classified by their discriminant, it shall suffice to prove that $disc(\ddagger_\lambda) = disc(\ddagger)$. But

	\begin{align*}
	\begin{pmatrix} 0 & 1 \\ -\lambda & 0 \end{pmatrix}^{\ddagger_\lambda} &= -\begin{pmatrix} 0 & 1 \\ -\lambda & 0 \end{pmatrix},
	\end{align*}
	
\noindent and the norm is
	
	\begin{align*}
	\det \begin{pmatrix} 0 & 1 \\ -\lambda & 0 \end{pmatrix} &= \lambda
	\end{align*}
	
\noindent and therefore

	\begin{align*}
	disc(\ddagger_\lambda) &= \lambda \left(F^\times\right)^2 = disc(\ddagger),
	\end{align*}
	
\noindent as claimed.
\end{proof}

With Lemma \ref{ExplicitAlgebra} in mind, we shall now rephrase the problem in terms of quadratic forms. Given any non-singular bilinear form $b$ on $V$, it is a simple computation that there is a unique involution $\ddagger_b$ of $End(V)$ defined by the property

	\begin{align*}
	b(v, \sigma w) &= b(\sigma^{\ddagger_b} v, w) \ \ \left(\forall v,w \in V, \sigma \in End(V)\right).
	\end{align*}
	
\noindent It is easily checked that scaling the bilinear form $b$ on $V$ does not change the involution $\ddagger_b$. However, it is a standard result (see \cite[p.1]{Involutions}) that $[b] \mapsto \ddagger_b$ defines a bijection between symmetric bilinear forms (defined up to scaling) and orthogonal involutions.

Thus, we henceforth fix a symmetric bilinear form $b_\ddagger$ on $V$ such that the corresponding involution is the orthogonal involution of $A$. We denote the associated quadratic form $q_\ddagger$. To prove Theorem \ref{MainTheorem}, we shall eventually need to consider the non-dyadic and dyadic places of $F$ separately. First, however, we give some results that hold generally.

\begin{lemma}\label{DescriptionOfIsomorphisms}
	\begin{align*}
	GO\left(End(V),\ddagger\right) = \bigcup_{l = 0}^1 \left\{g \in GL(V) \middle| q_\ddagger(gv) = (-1)^l \det(g) q_\ddagger(v) \ \forall v \in V\right\}.
	\end{align*}
\end{lemma}

\begin{proof}
Let $g \in GO\left(End(V),\ddagger\right)$ and $v \in V$. Then

	\begin{align*}
	q_\ddagger(gv) &= b_\ddagger(gv, gv) \\
	&= b_\ddagger\left(g^\ddagger g v, v\right) \\
	&= g^\ddagger g b_\ddagger(v,v) \\
	&= \pm \det(g) q_\ddagger(v),
	\end{align*}
	
\noindent hence

	\begin{align*}
	GO\left(End(V),\ddagger\right) \subset \bigcup_{l = 0}^1 \left\{g \in GL(V) \middle| q_\ddagger(gv) = (-1)^l \det(g) q_\ddagger(v) \ \forall v \in V\right\}.
	\end{align*}
	
\noindent On the other hand, if $q_\ddagger(gv) = \pm \det(g) q_\ddagger(v)$ for all $v \in V$, then $g$ preserves the quadratic form $q_\ddagger$ up to scaling, hence it preserves the bilinear form $b_\ddagger$ up to scaling. Since there is a bijective correspondence between bilinear forms (up to scaling) and orthogonal involutions on $End(V)$, we conclude that $g$ in fact preserves $\ddagger$.
\end{proof}

Next, to easily describe Eichler orders $\OO \cap \OO^\ddagger$ in $End(V)$, we shall need to consider $\mathfrak{o}$-lattices in the quadratic space $V$. For a given lattice $\Lambda \subset V$, we denote its dual lattice by

	\begin{align*}
	\Lambda^\sharp = \left\{v \in V \middle| b_\ddagger(v, \Lambda) \subset \mathfrak{o}\right\}.
	\end{align*}
	
\noindent Given a lattice $\Lambda$, we have a corresponding order in $End(V)$, defined by

	\begin{align*}
	End(\Lambda) &= \left\{\sigma \in End(V) \middle| \sigma(\Lambda) \subset \Lambda\right\}.
	\end{align*}
	
\noindent Note that scaling the lattice $\Lambda$ does not change the order.
	
\begin{lemma}\label{EichlerSurjection}
We have a map

	\begin{align*}
	\varphi: \left\{\substack{\text{lattices in } V \\ \text{up to scaling}}\right\} &\rightarrow \left\{\text{orders of } End(V)\right\} \\
	[\Lambda] &\mapsto End(\Lambda) \cap End(\Lambda^\sharp)
	\end{align*}
	
\noindent that surjects onto the subset of Eichler orders of the form $\OO \cap \OO^\ddagger$.
\end{lemma}

\begin{proof}
The proof is based on two observations. First, over a local field, $\mathfrak{o}$ is a PID, and it follows that all maximal orders are of the form $End(\Lambda)$ for some lattice $\Lambda$. Secondly, $End(\Lambda)^\ddagger = End(\Lambda^\sharp)$, since

	\begin{align*}
	End(\Lambda)^\ddagger &= \left\{\sigma \in End(V) \middle| \sigma^\ddagger\Lambda \subset \Lambda \right\} \\
	&= \left\{\sigma \in End(V) \middle| b\left(\Lambda^\sharp, \sigma^\ddagger \Lambda\right) \subset \mathfrak{o} \right\} \\
	&= \left\{\sigma \in End(V) \middle| b\left(\sigma \Lambda^\sharp, \Lambda\right) \subset \mathfrak{o} \right\} \\
	&= \left\{\sigma \in End(V) \middle| \sigma \Lambda^\sharp \subset \Lambda^\sharp \right\} \\
	&= End(\Lambda^\sharp).
	\end{align*}
\end{proof}

We recall the definitions of modular and maximal lattices. The scale, norm, and volume of a lattice $\Lambda$ are the fractional ideals

	\begin{align*}
	\mathfrak{s} \Lambda &= \bigcup_{v_1, v_2 \in \Lambda} b_\ddagger(v_1, v_2) \mathfrak{o}, \\
	\mathfrak{n} \Lambda &= \bigcup_{v \in \Lambda} q_\ddagger(v) \mathfrak{o}, \\
	\mathfrak{v} \Lambda &= \bigcup_{\Lambda = \mathfrak{o} v_1 \bot \mathfrak{o} v_2} \det\left(b_\ddagger(v_i, v_j)\right) \mathfrak{o},
	\end{align*}
	
\noindent respectively. Let $\mathfrak{a}$ be a fractional ideal of $F$. Then $\Lambda$ is $\mathfrak{a}$-modular if $\Lambda = \mathfrak{a}\Lambda^\sharp$---equivalently, if $\mathfrak{s}\Lambda = \mathfrak{a}$ and $\mathfrak{v}\Lambda = \mathfrak{a}^2$. Also, $\Lambda$ is $\mathfrak{a}$-maximal if $\mathfrak{n}\Lambda \subset \mathfrak{a}$ and for any other lattice $\Lambda' \supset \Lambda$

	\begin{align*}
	\mathfrak{n}\Lambda' \subset \mathfrak{a} \Rightarrow \Lambda = \Lambda'.
	\end{align*}

\noindent Isomorphisms of algebras with involution respect maximality---that is, they send maximal lattices to maximal lattices.
	
\begin{lemma}\label{MaximalOrdersPreserved}
Let $\phi \in PGO\left(End(V), \ddagger\right)$, and let $\Lambda$ be an $\mathfrak{a}$-maximal lattice. Then there exists $\lambda \in F^\times$ and a $\lambda \mathfrak{a}$-maximal lattice $\Lambda'$ such that

	\begin{align*}
	\phi\left(End(\Lambda) \cap End(\Lambda^\sharp)\right) &= End(\Lambda') \cap End({\Lambda'}^\sharp).
	\end{align*}
	
\noindent Conversely, if $\Lambda, \Lambda'$ are $\mathfrak{a}$-maximal lattices, then there exists $\phi \in PGO\left(End(V), \ddagger\right)$ such that

	\begin{align*}
	\phi\left(End(\Lambda) \cap End(\Lambda^\sharp)\right) &= End(\Lambda') \cap End({\Lambda'}^\sharp).
	\end{align*}
\end{lemma}

\begin{proof}
Choose a representative $g \in GO(H,\ddagger)$ such that $\phi(\sigma) = g \sigma g^{-1}$. Then

	\begin{align*}
	\phi\left(End(\Lambda) \cap End(\Lambda^\sharp)\right) &= g End(\Lambda)g^{-1} \cap g End(\Lambda)^\ddagger g^{-1} \\
	&= End(g\Lambda) \cap End(g\Lambda)^\ddagger \\
	&= End(g\Lambda) \cap End\left((g\Lambda)^\sharp\right).
	\end{align*}
	
\noindent It is straightforward to see that if $\Lambda$ is $\mathfrak{a}$-maximal, then $g\Lambda$ is $\det(g) \mathfrak{a}$-maximal.

On the other hand, if $\Lambda, \Lambda'$ are both $\mathfrak{a}$-maximal lattices over a local field, then there is an isomorphism of quadratic spaces $g: V \rightarrow V$ such that $\Lambda = g\Lambda'$. We have the desired the isomorphism

	\begin{align*}
	\phi(\sigma) = g\sigma g^{-1}.
	\end{align*}
\end{proof}

Finally, we also need the following technical lemma.

\begin{lemma}\label{TechnicalLemma}
Let $H = Mat(2,F)$ with the involution $\ddagger_\lambda$. Suppose $\Lambda$ is a lattice in $F^2$ with an orthogonal basis. Then there exists $g \in GO\left(Mat(2,F), \ddagger_\lambda\right)$ such that

	\begin{align*}
	End(\Lambda) \cap End(\Lambda^\sharp) \subset g \left(Mat(2,\mathfrak{o}) \cap Mat(2,\mathfrak{o})^\ddagger\right) g^{-1}.
	\end{align*}
\end{lemma}

\begin{proof}
We are given that $\Lambda = \mathfrak{o} e_1 \ \bot \ \mathfrak{o} e_2$. Therefore, we know that the Gram matrix $\left(b_\ddagger(e_i,e_j)\right)_{ij}$ can be assumed to be of the form

	\begin{align*}
	\begin{pmatrix} x & 0 \\ 0 & y \end{pmatrix}
	\end{align*}
	
\noindent Note that if $e_1 = (a,b)$, $e_2 = (c,d)$, then

	\begin{align*}
	\begin{pmatrix} x & 0 \\ 0 & y \end{pmatrix} &= \begin{pmatrix} a^2 \lambda + c^2 & c d + a b \lambda \\ cd + ab \lambda & b^2 \lambda + \mathfrak{o}^2 \end{pmatrix},
	\end{align*}
	
\noindent and so

	\begin{align*}
	\lambda x/y &= \frac{a^2 \lambda^2 + c^2\lambda}{b^2 \lambda + \mathfrak{o}^2} \\
	&= \frac{a^2 b^2 \lambda^2 + b^2 c^2\lambda}{b^2\left(b^2 \lambda + \mathfrak{o}^2\right)} \\
	&= \frac{c^2 \mathfrak{o}^2 + b^2 c^2\lambda}{b^2\left(b^2 \lambda + \mathfrak{o}^2\right)} \\
	&= \frac{c^2}{b^2} \in \left(F^\times\right)^2 \\
	\end{align*}
	
\noindent Consequently, the Gram matrix must be of the form

	\begin{align*}
	\begin{pmatrix} x & 0 \\ 0 & y \end{pmatrix} &= \begin{pmatrix} \lambda \pi^{2n} z \epsilon & 0 \\ 0 & z\end{pmatrix},
	\end{align*}
	
\noindent where $z \in \mathfrak{o}$, $\epsilon \in {\mathfrak{o}^\times}^2$ and $n \in \ZZ$. In fact, we can assume that $\epsilon = 1$, since we are free to scale $e_2$ by an element of $D^\times$, and this scales $q_\ddagger(e_2)$ by an element of ${\mathfrak{o}^\times}^2$.

However, the Gram matrix defines the lattice $\Lambda$ up to isomorphism of quadratic spaces, which we know lift to isomorphisms of $\ddagger$-orders. Therefore, it suffices to construct a single lattice for each Gram matrix. This is done explicitly by

	\begin{align*}
	\Lambda &= \mathfrak{o} \pi^n\left(a, -b\lambda\right) \oplus \mathfrak{o} \pi^n \left(b, a\right)
	\end{align*}
	
\noindent where

	\begin{align*}
	z &= a^2 + b^2 \lambda = q_\ddagger\left((a,b)\right).
	\end{align*}
	
\noindent If we define a matrix

	\begin{align*}
	L &= \begin{pmatrix} a & b \\ -b\lambda & a\end{pmatrix} \begin{pmatrix} \pi^n & 0 \\ 0 & 1 \end{pmatrix},
	\end{align*}
	
\noindent then

	\begin{align*}
	\sigma \in End(\Lambda) &\Leftrightarrow \sigma \in L Mat(2,\mathfrak{o}) L^{-1}.
	\end{align*}
	
\noindent However,

	\begin{align*}
	\begin{pmatrix} a & b \\ -b\lambda & a\end{pmatrix}^{\ddagger_\lambda}\begin{pmatrix} a & b \\ -b\lambda & a\end{pmatrix} &= \left(a^2 + b^2 \lambda\right) I
	\end{align*}
	
\noindent so

	\begin{align*}
	\begin{pmatrix} a & b \\ -b\lambda & a\end{pmatrix} \in GO\left(Mat(2,F), \ddagger_\lambda\right).
	\end{align*}
	
\noindent As we remarked earlier, conjugating by elements of $GO\left(Mat(2,F), \ddagger_\lambda\right)$ corresponds to isomorphism of $\ddagger$-orders, so we may assume that

	\begin{align*}
	L &= \begin{pmatrix} \pi^n & 0 \\ 0 & 1 \end{pmatrix},
	\end{align*}
	
\noindent and therefore

	\begin{align*}
	End(\Lambda) &= \begin{pmatrix} \pi^n & 0 \\ 0 & 1 \end{pmatrix} Mat(2,\mathfrak{o}) \begin{pmatrix} \pi^n & 0 \\ 0 & 1 \end{pmatrix}^{-1} \\
	End(\Lambda^\sharp) &= \begin{pmatrix} \pi^n & 0 \\ 0 & 1 \end{pmatrix}^{-1} Mat(2,\mathfrak{o})^\sharp \begin{pmatrix} \pi^n & 0 \\ 0 & 1 \end{pmatrix}.
	\end{align*}
	
\noindent We can then explicitly compute
	
	\begin{align*}
	End(\Lambda) \cap End(\Lambda^\sharp) &= \begin{cases} \left\{\begin{pmatrix} a & b \pi^{|n|} \\ c \pi^{|n| + 1} & d \end{pmatrix} \middle| a,b,c,d \in \mathfrak{o} \right\} & \text{if } \iota(disc(\ddagger)) = \mathfrak{p} \vspace{5pt} \\ \left\{\begin{pmatrix} a & b \pi^{|n|} \\ c \pi^{|n|} & d \end{pmatrix} \middle| a,b,c,d \in \mathfrak{o} \right\} & \text{if } \iota(disc(\ddagger)) = \mathfrak{o} \end{cases} \\
	&\subset Mat(2,\mathfrak{o}) \cap Mat(2,\mathfrak{o})^\ddagger.
	\end{align*}
\end{proof}

\section{The Non-Dyadic Case:}

We now specialize to the case where $F$ is a local field over a non-dyadic place (that is, $2 \in \mathfrak{o}^\times$). 

\begin{theorem}\label{NonDyadic}
Suppose that $A \cong End(V)$ is a split quaternion algebra over a local field $F$. Suppose $\OO$ is a maximal $\ddagger$-order. Then there exists a fractional ideal $\mathfrak{a} \subset \mathfrak{o}$ and an $\mathfrak{a}$-maximal lattice $\Lambda$ such that

	\begin{align*}
	\OO = End(\Lambda) \cap End(\Lambda^\sharp).
	\end{align*}
	
Furthermore, all maximal $\ddagger$-orders are isomorphic, and

	\begin{align*}
	disc(\OO) = disc(H) \cap \iota(disc(\ddagger)).
	\end{align*}
\end{theorem}

\begin{proof}
By Lemma \ref{ExplicitAlgebra}, we can assume that $H = Mat(2,F)$ and $\ddagger = \ddagger_\lambda$. In this case, the lattice

	\begin{align*}
	\mathfrak{o}^2 &= \mathfrak{o} \begin{pmatrix} 1 \\ 0 \end{pmatrix} + \mathfrak{o} \begin{pmatrix} 0 \\ 1 \end{pmatrix}
	\end{align*}
	
\noindent is $\mathfrak{o}$-maximal. This follows from the fact (see \cite{QuadraticForms}, p. 246) that if $\mathfrak{n}\Lambda \subset \mathfrak{a}$ and the ideal

	\begin{align*}
	2^2 \mathfrak{a}^{-2} \mathfrak{v}\Lambda
	\end{align*}
	
\noindent has no square factors, then $\Lambda$ is $\mathfrak{a}$-maximal. It is easy to see that $\mathfrak{n}\mathfrak{o}^2 = \mathfrak{o}$, and $\mathfrak{v}\mathfrak{o}^2 = \mathfrak{o}$ or $\mathfrak{p}$. In either case, the required ideal is square-free, and therefore $\mathfrak{o}^2$ is $\mathfrak{o}$-maximal.

We will show that all maximal $\ddagger$-orders are isomorphic to

	\begin{align*}
	Mat(2,\mathfrak{o}) \cap Mat(2,\mathfrak{o})^\ddagger &= \begin{cases} \left\{\begin{pmatrix} a & b \\ \pi c & d \end{pmatrix} \middle| a,b,c,d \in \mathfrak{o}\right\} & \text{if } \iota(disc(\ddagger)) = \mathfrak{p} \\ Mat(2,\mathfrak{o}) & \text{if } \iota(disc(\ddagger)) = \mathfrak{o}. \end{cases}
	\end{align*}

\noindent Note that this order has the desired discriminant and is the $\ddagger$-order corresponding to the maximal lattice $\mathfrak{o}^2$. By Lemma \ref{MaximalOrdersPreserved}, we will then have proved the theorem.

By Lemmas \ref{EichlerSurjection} and \ref{DescriptionOfIsomorphisms}, to show that all maximal $\ddagger$-orders are isomorphic to $Mat(2,\mathfrak{o}) \cap Mat(2,\mathfrak{o})^\ddagger$, it is enough to show that for every lattice $\Lambda \subset F^2$, there exists a $g \in GO\left(Mat(2,F), \ddagger_\lambda\right)$ such that

	\begin{align*}
	End(\Lambda) \cap End(\Lambda^\sharp) \subset g \left(Mat(2,\mathfrak{o}) \cap Mat(2,\mathfrak{o})^\ddagger\right) g^{-1}.
	\end{align*}
	
\noindent Since $F$ is non-dyadic, every lattice has an orthogonal basis, and so the result follows by Lemma \ref{TechnicalLemma}.
\end{proof}

\section{The Dyadic Case}

If $F$ is a local field over a dyadic place, the behavior of maximal $\ddagger$-orders is more complicated---unlike the non-dyadic case, it need not be true that $\mathfrak{o}^2$ is a maximal $\mathfrak{o}$-lattice.

\begin{lemma}\label{DefectOfLattice}
Let $F$ be a local field over a dyadic place. Let $\lambda \in F^\times$ be integral and square-free. The lattice $\mathfrak{o}^2$ is $\mathfrak{o}$-maximal for the quadratic form

	\begin{align*}
	q_\ddagger(x,y) &= x^2 + \lambda y^2
	\end{align*}
	
\noindent on $F^2$ if and only if $\mathfrak{d}(-\lambda) = \mathfrak{p}$.
\end{lemma}

\begin{proof}
It is easy to check that $\mathfrak{n}\mathfrak{o}^2 = \mathfrak{o}$, as desired. If $-\lambda$ is a square (that is, $q_\ddagger$ is isotropic), then it is easy to check that $\mathfrak{o}^2$ is not maximal---indeed, the lattice

	\begin{align*}
	\mathfrak{o}\begin{pmatrix} 1/2 \\ 1/2 \end{pmatrix} + \mathfrak{o} \begin{pmatrix} -1/2 \\ 1/2 \end{pmatrix}
	\end{align*}
	
\noindent clearly contains $\mathfrak{o}^2$ and yet also has norm $\mathfrak{o}$. Therefore, we can assume that $q_\ddagger$ is anisotropic. It follows that there is a unique maximal $\mathfrak{o}$-lattice, which is the set of all $x,y$ such that

	\begin{align*}
	q_\ddagger(x,y) \in \mathfrak{o}.
	\end{align*}
	
\noindent It is evident that $\mathfrak{o}^2$ is a sub-lattice---it remains to show that there are non-integral $x,y$ such that $q_\ddagger(x,y) \in \mathfrak{o}$ if and only if $\mathfrak{d}(-\lambda) \neq \mathfrak{p}$.

First, suppose there is such a pair $x,y \in F$. Let $n$ be the smallest integer such that $\pi^n x, \pi^n y \in \mathfrak{o}$---by assumption, $n \geq 1$. Therefore

	\begin{align*}
	(\pi^n x)^2 + \lambda(\pi^n y)^2 \in \mathfrak{p}^2,
	\end{align*}
	
\noindent and since it is easily seen that $y \neq 0$, we arrange to find that

	\begin{align*}
	-\lambda &= (x/y)^2 + \mathfrak{p}^2,
	\end{align*}
	
\noindent which shows $\mathfrak{d}(-\lambda) \neq \mathfrak{p}$. In the other direction, if $\mathfrak{d}(-\lambda) \neq \mathfrak{p}$, then we can find coprime $x,y \in D$ such that

	\begin{align*}
	-\lambda &= (x/y)^2 + \mathfrak{p}^2,
	\end{align*}

\noindent hence

	\begin{align*}
	x^2 + \lambda y^2 \in \mathfrak{p}^2,
	\end{align*}
	
\noindent in which case $(x/\pi, y/\pi) \notin \mathfrak{o}^2$ but do belong to the maximal $\mathfrak{o}$-lattice.
\end{proof}

Consequently, we get different behavior depending on whether $\mathfrak{d}\left(-disc(\ddagger) \cap \mathfrak{o}\right) = \mathfrak{p}$ or not.

\begin{theorem}\label{DyadicDMaximal}
Suppose $A \cong End(V)$ is a split quaternion algebra over $F$, a local field over a dyadic place $\mathfrak{p}$. Suppose that $\mathfrak{d}\left(-disc(\ddagger)\cap \mathfrak{o}\right) = \mathfrak{p}$. Let $\OO$ is a maximal $\ddagger$-order. Then there exists a fractional ideal $\mathfrak{a} \subset F$ and an $\mathfrak{a}$-maximal lattice $\Lambda$ such that

	\begin{align*}
	\OO = End(\Lambda) \cap End(\Lambda^\sharp).
	\end{align*}
	
\noindent Furthermore, all maximal $\ddagger$-orders are isomorphic, and

	\begin{align*}
	disc(\OO) = disc(H) \cap \iota\left(disc(\ddagger)\right).
	\end{align*}
\end{theorem}
	
\begin{proof}
The proof is very similar to the proof of Theorem \ref{NonDyadic}. We note that we can reduce to the special case $End(V) = Mat(2,F)$, $\ddagger = \ddagger_\lambda$, where $\lambda \mathfrak{o} = \mathfrak{o}$ or $\mathfrak{p}$. By Lemma \ref{DefectOfLattice}, $\mathfrak{o}^2$ is a maximal lattice. Ergo, if we prove that every maximal $\ddagger$-order $\OO$ of $Mat(2,F)$ is of the form

	\begin{align*}
	\OO = End(\Lambda) \cap End(\Lambda^\sharp),
	\end{align*}
	
\noindent where $\Lambda$ is a lattice in $F^2$ with orthogonal base, we will be done.

Over a local field, if $\mathfrak{n}\Lambda = \mathfrak{s}\Lambda$, there is an explicit algorithm to construct an orthogonal base of $\Lambda$. On the other hand, if $\Lambda = \mathfrak{o} e_1 + \mathfrak{o} e_2$, then

	\begin{align*}
	\mathfrak{s}\Lambda &= \sum_{i,j} q_\ddagger(e_i, e_j) \mathfrak{o} \\
	\mathfrak{n}\Lambda &= \sum_i q_\ddagger(e_i)\mathfrak{o} + 2 \mathfrak{s}\Lambda.
	\end{align*}

\noindent Let

	\begin{align*}
	\Lambda &= \mathfrak{o}\begin{pmatrix} a \\ c \end{pmatrix} + \mathfrak{o}\begin{pmatrix}b \\ d \end{pmatrix},
	\end{align*}
	
\noindent so that the Gram matrix is
	
	\begin{align*}
	G &= \begin{pmatrix} a & b \\ c & d \end{pmatrix}^T \begin{pmatrix} \lambda & 0 \\ 0 & 1 \end{pmatrix} \begin{pmatrix} a & b \\ c & d \end{pmatrix} \\
	&= \begin{pmatrix} \lambda a^2 + c^2 & \lambda ab + cd \\ \lambda ab + cd & \lambda b^2 + \mathfrak{o}^2 \end{pmatrix}.
	\end{align*}
	
\noindent Thus, the only way that $\mathfrak{n}\Lambda \neq \mathfrak{s}\Lambda$ is if

	\begin{align*}
	\min\{ord_\mathfrak{p}(\lambda a^2 + c^2),ord_\mathfrak{p}(\lambda b^2 + \mathfrak{o}^2)\} > ord_\mathfrak{p}(\lambda ab + cd).
	\end{align*}
	
\noindent However,

	\begin{align*}
	\left(\lambda a^2 + c^2\right)\left(\lambda b^2 + \mathfrak{o}^2\right) &= \left(\lambda ab + cd\right)^2 + \lambda\left(ad - bc\right)^2,
	\end{align*}
	
\noindent hence

	\begin{align*}
	-\lambda &= \lambda^2\left(ad - bc\right)^2 + \frac{\lambda\left(\lambda a^2 + c^2\right)\left(\lambda b^2 + \mathfrak{o}^2\right)}{\left(\lambda ab + cd\right)^2} \\
	&= \lambda^2\left(ad - bc\right)^2 + \mathfrak{p}^2,
	\end{align*}
	
\noindent which is impossible since we are given that $\mathfrak{d}\left(-\lambda\right) = \mathfrak{p}$. Therefore, $\Lambda$ has an orthogonal base, and we are done.
\end{proof}

If $\mathfrak{d}\left(-disc(\ddagger) \cap \mathfrak{o}\right) \neq \mathfrak{p}$, we shall show in the next section that there are always at least two isomorphism classes of maximal $\ddagger$-orders. For now, we give a weaker result.

\begin{theorem}\label{DyadicDNotMaximal}
Suppose $A \cong End(V)$ is a split quaternion algebra over $F$, a local field over a dyadic place $\mathfrak{p}$. Suppose that $\mathfrak{d}\left(-disc(\ddagger)\cap \mathfrak{o}\right) \neq \mathfrak{p}$. Let $\OO$ is a maximal $\ddagger$-order. Then there exists a fractional ideal $\mathfrak{a} \subset F$ and an $\mathfrak{a}$-modular lattice $\Lambda$ such that

	\begin{align*}
	\OO = End(\Lambda) \cap End(\Lambda^\sharp).
	\end{align*}
	
\noindent Additionally,

	\begin{align*}
	disc(\OO) = disc(H) \cap \iota(disc(\ddagger)) = \mathfrak{o}.
	\end{align*}
\end{theorem}

\begin{proof}
Since $\mathfrak{d}\left(-disc(\ddagger)\cap \mathfrak{o}\right) \neq \mathfrak{p}$, we conclude that $\iota\left(disc(\ddagger)\right) = \mathfrak{o}$---otherwise, there would be an element $\lambda \in disc(\ddagger)$ such that $ord_\mathfrak{p}\left(-\lambda\right) = 1$, which would imply that $\mathfrak{d}\left(-disc(\ddagger)\cap \mathfrak{o}\right) = \mathfrak{p}$.

By Lemma \ref{ExplicitAlgebra}, we can assume that $H = Mat(2,F)$ and $\ddagger = \ddagger_\lambda$. Since $\iota\left(disc(\ddagger)\right) = \mathfrak{o}$, $\mathfrak{o}^2$ is a unimodular lattice. Choose any maximal $\ddagger$-order $\OO = End(\Lambda) \cap End(\Lambda^\sharp)$.

By the theory of Jordan splittings, $\Lambda$ is either modular or it has an orthogonal base. However, if it has an orthogonal base, then we know by \ref{TechnicalLemma} that $\OO$ is isomorphic to $Mat(2,\mathfrak{o}) \cap Mat(2,\mathfrak{o})^\ddagger$---since $\mathfrak{o}^2$ is unimodular, we see that we can always take $\Lambda$ to be modular.

Since $\Lambda$ is $\mathfrak{a}$-modular for some fractional ideal $\mathfrak{a}$, $\Lambda = \mathfrak{a}\Lambda^\sharp$. This proves that $\OO = End(\Lambda)$, hence its discriminant is $\mathfrak{o}$.
\end{proof}

\section{Counting Orders over Local Fields:}

In this section, we shall give methods for precisely counting the number of maximal $\ddagger$-orders in a split quaternion algebra over a local field, as well as methods of counting the number of isomorphism classes. To start, we shall need the following lemma.

\begin{lemma}\label{StabilizerLemma}
Let $F$ be a local field. Then the stabilizer of $Mat(2,\mathfrak{o})$ in $GL(2,F)$ is

	\begin{align*}
	GL^0(2,D) &= \ker\left(GL(2,F) \stackrel{\det}{\rightarrow} F^\times \rightarrow F^\times / \left(F^\times\right)^2\right) \cap Mat(2,\mathfrak{o})
	\end{align*}
\end{lemma}

\begin{proof}
Suppose that $\gamma = \left(\begin{smallmatrix} a & b \\ c & d \end{smallmatrix}\right) \in Stab\left(Mat(2,\mathfrak{o})\right)$. Then

	\begin{align*}
	\begin{pmatrix} a & b \\ c & d \end{pmatrix} \begin{pmatrix} 0 & 1 \\ 0 & 0 \end{pmatrix} \begin{pmatrix} a & b \\ c & d \end{pmatrix}^{-1} &\in Mat(2,\mathfrak{o}), \\
	\begin{pmatrix} a & b \\ c & d \end{pmatrix} \begin{pmatrix} 0 & 0 \\ 1 & 0 \end{pmatrix} \begin{pmatrix} a & b \\ c & d \end{pmatrix}^{-1} &\in Mat(2,\mathfrak{o}),
	\end{align*}
	
\noindent and therefore

	\begin{align*}
	a^2, b^2, c^2, \mathfrak{o}^2 &\in \det(\gamma) \mathfrak{o}.
	\end{align*}
	
\noindent If $\gamma \in Stab\left(Mat(2,\mathfrak{o})\right)$, then so is $\lambda \gamma$. Therefore, we are free to assume that either $\det(\gamma)\mathfrak{o} = \mathfrak{o}$ or $\mathfrak{p}$. Suppose first that it is $\mathfrak{p}$. Then	$a,b,c,d \in \mathfrak{p}$, hence $ad, bc \in \mathfrak{p}^2$. However, this implies $ad - bc = \det(\gamma) \in \mathfrak{p}^2$, which is a contradiction.

Therefore, we are free to assume that $\det(\gamma) \in \mathfrak{o}^\times$, which proves that $a,b,c,d \in \mathfrak{o}$. What we have shown is that

	\begin{align*}
	\gamma \in \begin{pmatrix} \pi^n & 0 \\ 0 & \pi^n \end{pmatrix} GL(2,\mathfrak{o}) \subset GL^0(2,\mathfrak{o}).
	\end{align*}
	
\noindent It is easy to see that $GL^0(2,\mathfrak{o}) \subset Stab\left(Mat(2,\mathfrak{o})\right)$, so we are done.
\end{proof}

\begin{corollary}\label{MaximalUniqueness}
Given two lattices $\Lambda_1, \Lambda_2 \subset V$, $End(\Lambda_1) = End(\Lambda_2)$ if and only if $\Lambda_1 = \lambda\Lambda_2$ for some $\lambda \in F^\times$.
\end{corollary}

\begin{proof}
We showed previously that if $L_1, L_2$ are matrices corresponding to bases of $\Lambda_1, \Lambda_2$, then

	\begin{align*}
	End(\Lambda_1) &= L_1 Mat(2,\mathfrak{o}) L_1^{-1} \\
	End(\Lambda_2) &= L_2 Mat(2,\mathfrak{o}) L_2^{-1},
	\end{align*}
	
\noindent and therefore
	
	\begin{align*}
	End(\Lambda_1) = End(\Lambda_2) &\Leftrightarrow L_1^{-1} L_2 \in GL^0(2,\mathfrak{o}) \\
	&\Leftrightarrow L_2 = L_1 \gamma \ \left(\text{for some } \gamma \in GL^0(2,\mathfrak{o})\right) \\
	&\Leftrightarrow L_2 = \lambda L_1 \gamma \ \left(\text{for some } \gamma \in GL(2,\mathfrak{o}), \ \lambda \in F^\times\right).
	\end{align*}
	
\noindent However, multiplying $L_1$ by an element of $GL_2(\mathfrak{o})$ does not change the lattice $\Lambda_1$ (this amounts to a change of basis), hence the result follows.
\end{proof}

\begin{corollary}\label{EichlerUniqueness}
Given two Eichler orders

	\begin{align*}
	\OO_1 &= End(\Lambda_1) \cap End(\Lambda_1^\sharp) \\
	\OO_2 &= End(\Lambda_2) \cap End(\Lambda_2^\sharp),
	\end{align*}
	
\noindent $\OO_1 = \OO_2$ if and only if $\Lambda_1 = \lambda \Lambda_2$ or $\Lambda_1 = \lambda \Lambda_2^\sharp$ (for some $\lambda \in F^\times$).
\end{corollary}

\begin{proof}
This follows immediately from the preceding corollary and the observation that, over a local field, the maximal orders of which an Eichler order is an intersection are uniquely determined.
\end{proof}

Motivated by Corollary \ref{EichlerUniqueness}, we define an equivalence relation on the set of lattices in $V$ by

	\begin{align*}
	\Lambda_1 \sim \Lambda_2 \Leftrightarrow \Lambda_1 = \lambda\Lambda_2 \text{ or } \Lambda_1 = \lambda \Lambda_2^\sharp \ \left(\text{for some } \lambda \in F^\times\right).
	\end{align*}
	
\noindent From Corollary \ref{EichlerUniqueness}, it follows that the set of equivalence classes of lattices is in bijection with the set of $\ddagger$-orders of $End(V)$ that are Eichler orders. We can give an even stronger statement, although we need to split into two cases

\begin{theorem}\label{NormalCount}
Let $V$ be a 2-dimensional vector space over a local field $F$ with maximal ideal $\mathfrak{p}$. Suppose either $2 \notin \mathfrak{p}$ or $\mathfrak{d}(-disc(\ddagger) \cap \mathfrak{o}) = \mathfrak{p}$. Then there is a well-defined bijection

	\begin{align*}
	\varphi: \left\{\text{maximal lattices in } V\right\} / \sim &\rightarrow \left\{\text{maximal } \ddagger\text{-orders of } End(V)\right\} \\
	[\Lambda] &\mapsto End(\Lambda) \cap End(\Lambda^\sharp).
	\end{align*}
\end{theorem}

\begin{proof}
As usual, we use Lemma \ref{ExplicitAlgebra} to reduce to the case where $A = Mat(2,F)$, $\ddagger = \ddagger_\lambda$.

To show $\varphi$ is well-defined, we need to show that if $\Lambda$ is a maximal lattice, then $\varphi([\Lambda])$ is a maximal $\ddagger$-order. We can freely assume that $\Lambda$ is either $\mathfrak{o}$-maximal or $\mathfrak{p}$-maximal. Over a local field, all $\mathfrak{a}$-maximal lattices are isomorphic---since $\mathfrak{o}^2$ is maximal and $\varphi([\mathfrak{o}^2])$ is a maximal $\ddagger$-order, we can conclude that any $\mathfrak{o}$-maximal lattice corresponds to a maximal $\ddagger$-order. So, it remains to consider $\mathfrak{p}$-maximal orders.

If $(F^2,q_\ddagger)$ is isotropic, there is a $g \in GO\left(Mat(2,F), \ddagger\right)$ such that $\det(g) = \pi$, and therefore $g\mathfrak{o}^2$ is a $\mathfrak{p}$-maximal lattice corresponding to a maximal $\ddagger$-order. Ergo all $\mathfrak{p}$-maximal orders correspond to maximal $\ddagger$-orders. If $(F^2,q_\ddagger)$ is anisotropic, there is a unique $\mathfrak{p}$-maximal lattice, given by

	\begin{align*}
	\Lambda_\mathfrak{p} &= \left\{v \in F^2 \middle| q_\ddagger(v) \in \mathfrak{p}\right\}.
	\end{align*}
	
\noindent There are two cases: either there exists $v \in F^2$ such that $q_\ddagger(v) \in \mathfrak{p} \backslash \mathfrak{p}^2$, or there doesn't. If there is such a $v$, then there is a $g \in GO\left(Mat(2,F), \ddagger\right)$ such that $\det(g) = q_\ddagger(v)$, and therefore $g\mathfrak{o}^2$ is a $\mathfrak{p}$-maximal lattice---the result of the argument is the same as in the isotropic case. If there is no such $v$, then $\Lambda_\mathfrak{p}$ is $\mathfrak{p}^2$-maximal, hence $\frac{1}{\pi}\Lambda_\mathfrak{p} = \mathfrak{o}^2$, and so we conclude that $\varphi([\Lambda_\mathfrak{p}])$ is a maximal $\ddagger$-order.

That $\varphi$ is injective follows from Corollary \ref{EichlerUniqueness}. Surjectivity follows from Theorems \ref{NonDyadic} and \ref{DyadicDMaximal}.
\end{proof}

\begin{theorem}\label{StrangeCount}
Let $V$ be a 2-dimensional vector space over a local field $F$ over a dyadic place $\mathfrak{p}$ and $\mathfrak{d}(-disc(\ddagger) \cap \mathfrak{o}) \neq \mathfrak{p}$. Then there is a well-defined bijection

	\begin{align*}
	\varphi: \left\{\text{modular lattices in } V\right\} / \sim &\rightarrow \left\{\text{maximal } \ddagger\text{-orders of } End(V)\right\} \\
	[\Lambda] &\mapsto End(\Lambda) \cap End(\Lambda^\sharp) = End(\Lambda).
	\end{align*}
	
\noindent Furthermore, two maximal $\ddagger$-orders $\varphi([\Lambda_1]), \varphi([\Lambda_2])$ are isomorphic if and only if there are representatives $\Lambda_1 \in [\Lambda_1]$, $\Lambda_2 \in [\Lambda_2]$ and an element $g \in GO\left(End(V), \ddagger\right)$ such that $\Lambda_1 = g\Lambda_2$.

Finally, the number of isomorphism classes of maximal $\ddagger$-orders is

\begin{align*}
\begin{cases} m + 1 & \text{if } \mathfrak{d}(-\lambda) = \mathfrak{p}^{2m + 1} \\ n + 1 & \text{if } \mathfrak{d}(-\lambda) = (4) \text{ or } 0 \end{cases},
\end{align*}

\noindent where $n = ord_\mathfrak{p}(2)$.
\end{theorem}

\begin{proof}
If $\Lambda$ is $\mathfrak{a}$-modular, then $End(\Lambda^\sharp) = End(\mathfrak{a}\Lambda) = End(\Lambda)$, hence $End(\Lambda)$ is a maximal $\ddagger$-order. Therefore, $\varphi$ is well-defined.

That $\varphi$ is injective follows from Corollary \ref{EichlerUniqueness}. Surjectivity follows from Theorem \ref{DyadicDNotMaximal}. The fact that isomorphism of maximal $\ddagger$-orders corresponds to acting on the lattices by $GO\left(End(V), \ddagger\right)$ is an easy consequence of bijectivity and Lemma \ref{DescriptionOfIsomorphisms}.

It remains to determine the number of non-isomorphic classes of maximal $\ddagger$-orders. To do this, we make use of the norm group. For a lattice $\Lambda$, its norm group is defined as

	\begin{align*}
	\mathfrak{g}\Lambda = q_\ddagger(\Lambda) + 2 \mathfrak{s}\Lambda,
	\end{align*}
	
\noindent or equivalently as

	\begin{align*}
	\mathfrak{g}\Lambda = a \mathfrak{o}^{sq} + b \mathfrak{o},
	\end{align*}
	
\noindent where $a$ is a norm generator, $b$ is a weight generator, and

	\begin{align*}
	\mathfrak{o}^{sq} = \left\{\mathfrak{o}^2\middle| d \in \mathfrak{o}\right\}.
	\end{align*}
	
\noindent Now, notice that since $GO\left(End(V), \ddagger\right)$ contains an element with determinant in $\mathfrak{p}\backslash \mathfrak{p}^2$, each isomorphism class contains at least one unimodular representative. Two unimodular lattices over a dyadic field are isomorphic if and only the norm groups are isomorphic. Therefore, two unimodular lattices $\Lambda_1, \Lambda_2$ correspond to the same isomorphism class if and only if there is an element $\gamma\in GO(H,\ddagger)$ such that

	\begin{align*}
	\det(\gamma) &\in \mathfrak{o}^\times \\
	\mathfrak{g}\Lambda_1 &= \gamma^\ddagger \gamma \mathfrak{g}\Lambda_2.
	\end{align*}
	
\noindent Let the norm generators of $\Lambda_1,\Lambda_2$ be $a_1, a_2$ respectively, and let the weight generators be $b_1,b_2$. Then $\Lambda_1,\Lambda_2$ are in the same isomorphism class if and only if there exist $s,t \in F$ such that $s^2 + t^2 \lambda \in \mathfrak{o}^\times$ and

	\begin{align*}
	a_1 \mathfrak{o}^{sq} + b_1 \mathfrak{o} &= \pm(s^2 + t^2\lambda) a_2 \mathfrak{o}^{sq} + b_2 \mathfrak{o}.
	\end{align*}
	
\noindent However, over a dyadic field, any unimodular lattice has a basis in which the Gram matrix is

	\begin{align*}
	\begin{pmatrix} \alpha & 1 \\ 1 & \beta \end{pmatrix},
	\end{align*}

\noindent where $\alpha, \beta \in \mathfrak{o}$, $-1 + \alpha\beta \in \mathfrak{o}^\times$, and $\alpha$ is a norm generator (in particular, $ord_\mathfrak{p}(\alpha) < ord_\mathfrak{p}(\beta)$). This implies that there exist $s_1, s_2, t_1, t_2 \in F$ such that

	\begin{align*}
	a_1 &= s_1^2 + t_1^2 \lambda \\
	a_2 &= s_2^2 + t_2^2 \lambda,
	\end{align*}
	
\noindent and therefore there exists $s,t$ such that $s^2 + t^2\lambda \in \mathfrak{o}^\times$ and $a_1 = (s^2 + t^2\lambda)a_2$ if and only if 

	\begin{align*}
	ord_\mathfrak{p}(a_1) = ord_\mathfrak{p}(a_2).
	\end{align*}
	
\noindent In other words, we see that the isomorphism class of $\Lambda_1, \Lambda_2$ is wholly determined by the order of their norm and weight generators.

However, choosing a basis as above such that the Gram matrix is

	\begin{align*}
	\begin{pmatrix} \alpha & 1 \\ 1 & \beta \end{pmatrix},
	\end{align*}
	
\noindent we get that the norm generator is $\alpha$ and the weight generator is either $\beta$ or $2$ (if $ord_\mathfrak{p}(\beta) > n$).

There exists a unimodular lattice $\Lambda$ with Gram matrix

	\begin{align*}
	\begin{pmatrix} \alpha & 1 \\ 1 & \beta \end{pmatrix}
	\end{align*}
	
\noindent if and only if

	\begin{align}\label{DefiningEquation}
	\det\begin{pmatrix} -\lambda & 0 \\ 0 & 1 \end{pmatrix} &= -\lambda = -1 + \alpha\beta \in \mathfrak{o}^\times/\left(\mathfrak{o}^\times\right)^2.
	\end{align}
	
\noindent Note that $\lambda$ is only defined up to multiplication by $\left(\mathfrak{o}^\times\right)^2$, so we can take it to be $-1 + \pi^{2m + 1}\epsilon$, $-1 + 4\epsilon$, or $-1$ (for some $\epsilon \in \mathfrak{o}^\times$). Furthermore, since we can scale $\Lambda$ by elements of $\mathfrak{o}^\times$, we can assume that we have equality in $\mathfrak{o}$ itself in equation \ref{DefiningEquation}.

First, suppose that $\lambda = -1 + \pi^{2m + 1}\epsilon$. We have that

	\begin{align*}
	-1 + \pi^{2m + 1}\epsilon &= -1 + \alpha\beta,
	\end{align*}
	
\noindent showing that

	\begin{align*}
	ord_\mathfrak{p}(\alpha) + ord_\mathfrak{p}(\beta) = 2m + 1 < 2n.
	\end{align*}
	
\noindent Thus, $\beta$ is the weight generator if $ord_\mathfrak{p}(\alpha) \geq 2m - n + 1$. Otherwise, $2$ is the weight generator. Therefore, all possible choices for the orders of the norm and weight generators are

	\begin{align*}
	\begin{cases} (m,m + 1),(m - 1, m + 2), \ldots (2m - n + 1, n), (2m - n, n), (2m - n - 1, n), \ldots (0,n) & \text{if } 2m + 1 > n \\ (m, m + 1),(m - 1, m + 2), \ldots (0, 2m + 1) & \text{if } 2m + 1 \leq n \end{cases}.
	\end{align*}

\noindent Second, suppose that $\lambda = -1 + 4\epsilon$. We have that

	\begin{align*}
	ord_\mathfrak{p}(\alpha) + ord_\mathfrak{p}(\beta) = 2n,
	\end{align*}
	
\noindent hence $\beta$ is the weight generator if $ord_\mathfrak{p}(\alpha) \geq n$, and otherwise $2$ is the weight generator. Therefore, all possible choices for the orders of the norm and weight generators are

	\begin{align*}
	(n,n), (n - 1, n), (n - 2, n), \ldots (0,n).
	\end{align*}
	
\noindent Finally, suppose that $\lambda = -1$. We have that $\alpha\beta = 0$. Since $\alpha$ is taken to be norm generator, it is non-zero---therefore, $\beta = 0$, and therefore the weight generator is simply $2$. The norm generator cannot have order greater than $n$, and therefore we see that all the possible choices for the orders of the norm and weight generators are again

	\begin{align*}
	(n,n), (n - 1, n), (n - 2, n), \ldots (0,n).
	\end{align*}
	
\noindent This proves the claim.
\end{proof}

\hskip0.1in

\bibliography{Orders}

\begin{thebibliography}{KMRT98}

\bibitem[IR93]{IvanyosRonyai}
G\'{a}bor Ivanyos and Lajos R\'{o}nyai.
\newblock Finding maximal orders in semisimple algebras over $\mathbb{Q}$.
\newblock {\em Computational Complexity}, 3(3):245--261, 1993.

\bibitem[KMRT98]{Involutions}
Max-Albert Knus, Alexander Merkurjev, Markus Rost, and Jean-Pierre Tignol.
\newblock {\em The Book of Involutions}.
\newblock American Mathematical Society, June 1998.

\bibitem[O'M73]{QuadraticForms}
O.~Timothy O'Meara.
\newblock {\em Introduction to Quadratic Forms}.
\newblock Springer, 1973.

\bibitem[She]{Self}
Arseniy Sheydvasser.
\newblock Applications of quaternion algebras to sphere packings.
\newblock in preparation.

\bibitem[Sta14]{Stange1}
Katherine~E. Stange.
\newblock Visualising the arithmetic of imaginary quadratic fields, 2014.

\bibitem[Sta15]{Stange2}
Katherine~E. Stange.
\newblock The apollonian structure of bianchi groups, 2015.

\bibitem[Voi13]{VoightArticle}
John Voight.
\newblock Identifying the matrix ring: Algorithms for quaternion algebras and
  quadratic forms.
\newblock {\em Quadratic and Higher Degree Forms Developments in Mathematics},
  pages 255--298, 2013.

\end{thebibliography}
\bibliographystyle{alpha}

\hskip0.1in

\address{Department of Mathematics, Yale University, 10 Hillhouse Avenue, New Haven, CT 06511}

\email{arseniy.sheydvasser@yale.edu}
\end{document}